\documentclass[11pt]{amsart}
\usepackage{color}

\numberwithin{equation}{section}
\newtheorem{theorem}{Theorem}

\newtheorem{proposition}{Proposition}
\newtheorem{lemma}{Lemma}
\newtheorem{corollary}{Corollary}
\numberwithin{theorem}{section} \numberwithin{lemma}{section}
 \numberwithin{definition}{section}
\numberwithin{proposition}{section}
\newtheorem{remark}{Remark}[section]
\def\R{\mathbb R}
\def\bR{\mathbb R}

\def\al{\aligned}
\def\eal{\endaligned}
\def\M{{\bf M}}
\def\be{\begin{equation}}
\def\ee{\end{equation}}
\def\lab{\label}

\def\al{\aligned}
\numberwithin{equation}{section}

\begin{document}

\tracingpages 1
\title[time analyticity]{Time analyticity for the heat equation
and Navier-Stokes equations}

	\author[Hongjie Dong]{Hongjie Dong$^*$}	
	
	\address{Division of Applied Mathematics, Brown University,
Providence, RI 02912, USA}
	
	\email{hongjie$\_$dong@brown.edu}

\author[Qi S. Zhang]{Qi S. Zhang}
\address{Department of Mathematics, University of California,
Riverside, CA 92521, USA }

\email{qizhang@math.ucr.edu}

\date{\today}

\begin{abstract}

We prove the analyticity in time for solutions of two parabolic equations in the whole space, without any decaying or vanishing conditions. One of them involves solutions to the heat equation of   exponential growth of order $2$ on $\M$. Here $\M$ is $\R^d$ or a complete noncompact manifold with Ricci curvature bounded from below by a constant. An implication is a sharp solvability condition for the Cauchy problem of the backward heat equation, which is a well known ill-posed problem. Another implication is a sharp criteria for time analyticity of solutions down to the initial time. The other pertains bounded mild solutions of the incompressible Navier-Stokes equations in the whole space.

There are many long established analyticity results for the Navier-Stokes equations.  See for example \cite{Ka:1} and \cite{FT:1} for spatial and time analyticity in certain integral sense, \cite{CN:1} for pointwise space-time analyticity of 3 dimensional solutions to the Cauchy problem,   and also the pointwise time analyticity results of \cite{Ma:1} and \cite{Gi:1} under zero boundary condition. Our result seems to be the first general pointwise time analyticity result for the Cauchy problem for all dimensions, whose proof involves only real variable method. The proof involves a method of algebraically manipulating the integral kernel, which appears applicable to other evolution equations.

\end{abstract}
\maketitle
%\tableofcontents

\section{Introduction}

In the study of heat and other parabolic equations, one often
hopes to prove that solutions are real analytic in space and
time. While the spatial analyticity is usually
true for generic
solutions, the time analyticity is harder to prove
and is false in
general. For example, it is not difficult to construct a
solution of the heat equation in a finite space-time cylinder
in
the Euclidean setting, which is not time analytic in a sequence
of moments. On the other hand,
under extra assumptions, many time-analyticity
results for the heat equation can be found in the literature.
See, for example,  \cite{Wi:1}.  Moreover, if one imposes zero
boundary conditions on the lateral boundary of a smooth
cylindrical domain, then certain solutions of the heat,
Navier-Stokes, and many other parabolic equations are analytic
in time. See, for example, \cite{Ma:1}, \cite{Ko},
\cite{Gi:1}, and
\cite{EMZ:1}.  One can also consider solutions in certain $L_p$
spaces with $p \in (1, \infty)$. In this setting,
by using complexification argument the time
analyticity with values in an $L_2$-based Gevrey
class of periodic functions
was proved for the Navier-Stokes equations in
\cite{FT:1}. See also \cite{Pr} for an extension to
a large class of dissipative equations in the periodic
setting.

In a related development, there have been renewed interest in the
study of global solutions of the heat equation on the Euclidean
and manifold setting. One example is the study of ancient
solutions of the heat equation, i.e., solutions that exist in
the whole space and in all negative time. Let $\M=\R^d$ or a
complete noncompact Riemannian manifold with nonnegative Ricci
curvature.
In \cite{SZ:1}, it was found that sublinear ancient solutions
are constants.  Later in \cite{LZ:1}, it was shown that the
space of ancient solutions of polynomial growth has finite
dimension and the solutions are polynomials in time. Colding
and
Minicozzi \cite{CM:1} then obtained a sharp dimension bound of
this space. See also  the papers \cite{Ca:1}
and \cite{CM:2} for applications to the study of mean curvature
flows, and  \cite{Ha:1} in the graph case.
In a recent paper \cite{Z:1}, it was observed that ancient
solutions on the above $\M$ with exponential growth in the
space
variable are analytic in time. One application of this result
is
a necessary and sufficient condition on the solvability of the
backward heat equation in this class of solutions, which is
well
known to be ill-posed in general.
Backward heat equations have been studied by many
authors, see, for example, \cite{Mi:1} and \cite{Yo:1}; and
treated in many books. See, for instance, \cite{LL:1}.  They have been applied in such diverse fields as control theory, stochastic analysis,
Ricci flows etc.
There does not appear to be a necessary and sufficient
solvability criteria, except when
$\M$ is a bounded domain for which semigroup theory gives an
abstract criteria. See \cite[Theorem 9]{CJ:1}.

One goal of the current paper is to show that the result in
\cite{Z:1} can be extended to solutions with  exponential
growth  of order $2$ (see \ref{expg}), which is a sharp condition.  Another  goal
is to prove the time analyticity for
all bounded mild solutions of the
Navier-Stokes equations in the whole space. One
implication is that bounded mild
solutions of the Navier-Stokes equations are analytic in
space-time, which yields the unique continuation
property of such solutions. This result may also
have applications in the study of
possible singularity whose blow up limit can be such a
solution.
We note that the space analyticity in more general
setting has been
proven in \cite{Ka:1, GK:1, GS:1, MS, GPS, DL, BBT, Gu, CKV:1, Xu:1},
to name a few.

In a subsequent work, we will study the corresponding problems in the half space.

We will present and prove the results for the heat equation and
the Navier-Stokes equations in Sections 2 and 3 respectively.

\section{The heat equation}

Let $\M$ be a $d$ dimensional, complete, noncompact Riemannian
manifold, $Ric$ be the Ricci curvature and $0$ be a reference
point on $\M$, $d(x, y)$ be the geodesic distance of $x, y \in
\M$. We will use $B(x, r)$  to denote the geodesic ball of
radius $r$ centered at $x$ and
$|B(x, r)|$ to denote the volume.  Given a point $(t,x)$ as
vertex, the standard parabolic cylinder of size $r$ is
$$
Q_r(t,x)=\{ (s,y)  \, | \, d(x, y)<r,  s \in [t-r^2, t] \}.
$$

\begin{theorem}
\lab{thhepoly}
 Let $\M$ be a complete, $d$ dimensional, noncompact Riemannian
 manifold such that the Ricci curvature satisfies $Ric \ge
 -(d-1) K_0$ for a nonnegative constant $K_0$.

Let $u$ be a smooth solution of the heat equation $\partial_t
u-\Delta u=0$ on $[-2, 0] \times {\M}$ of   exponential growth of order $2$, namely
\be
\lab{expg}
 |u{(t,x)}| \le A_1 e^{ A_2  d^2(x, 0)}, \quad \forall (t,x)
 \in
 [-2, 0] \times {\M},
\ee where $A_1$ and $A_2$ are positive constants.
Then $u=u(t,x)$ is analytic in $t \in [-1, 0]$ with radius
$r>0$  depending only on $d$, $K_0$, and $A_2$. Moreover, we
have
$$
u(t,x)= \sum^\infty_{j=0} a_j(x) \frac{ t^j}{j !}
$$
with $\Delta a_{j}(x) = a_{j+1}(x)$, and
$$
|a_{j}(x)| \le
A_1 \, A^{j+1}_3 j^j e^{2 A_2  d^2(x, 0)}, \quad
j=0,1,2,\ldots,
$$
where  $A_3$ is a positive constants depending only on $d$,
$K_0$, and $A_2$.
\end{theorem}

\begin{proof}
Since the equation is linear, without loss of generality, we
may
assume that $A_1=1$.
It suffices to prove the result for the space time point
$(0,x)$.
Let us recall a well-known parabolic mean value inequality
which
can be found, for instance, in \cite[Theorem 14.7]{Li:1}.
Suppose $v$ is a positive subsolution to the heat equation on
$[0, T] \times {\M}$ and $0$ is a point on $\M$. Let $T_1, T_2
\in [0, T]$ with $T_1<T_2$, $R>0$, $p>0$, and $\delta, \eta \in
(0, 1)$. Then there exist  positive constants $C_1$ and $C_2$,
depending only on $p$ and $d$, such that
$$
\al
&\sup_{[T_1, T_2]\times B(0, (1-\delta) R) }  v^p\\
&\le C_1
\frac{\bar{V}(2R)}{|B(0, R)|}
( R \sqrt{K_0} + 1) \exp ( C_2 \sqrt{K_0 (T_2-T_1)}) \\
&\qquad \times
\left(\frac{1}{\delta R} + \frac{1}{\sqrt{\eta T_1}} \right)^{d+2} \,
\int^{T_2}_{(1-\eta) T_1} \int_{B(0, R)} v^p(t,x)\ dx dt.
\eal
$$
Here $\overline{V}(R)$ is the volume of geodesic balls of radius $R$
in the simply connected space form with constant sectional
curvature $-K_0$.

Let $u$ be the given solution to the heat equation, so that
$u^2$ is a subsolution. Given $x_0 \in \M$ and a
positive integer $k$, with a translation of time, the  above
mean value inequality with $T_1=-1/k$, $T_2=0$, $R=1/\sqrt k$,
$\eta=\delta=1/2$ implies that
$$
\al
\sup_{Q_{1/(2 \sqrt{k})} (0,x_0)} u^2 &\le \frac{C_1
k^{d+2}}{|B(x_0, 1/\sqrt{k})|}
\int_{Q_{1/\sqrt{k}} (0,x_0)} u^2(t,x) \ dx dt\\
&\le \frac{C_2 k^{3d/2+2}}{|B(x_0, 1)|}
\int_{Q_{1/\sqrt{k}} (0,x_0)} u^2(t,x) \ dx dt,
\eal
$$
where we have used the Bishop-Gromov volume comparison
theorem. Note that the above mean value inequality is a local
one since the size of the cubes is less than one. Hence the
constants $C_1$ and $C_2$ are independent of $k$. Since
$\partial^k_t u$ is also a solution to the heat equation, it
follows that
\be
\lab{mviqdkt}
\sup_{Q_{1/(2 \sqrt{k})} (0,x_0)} (\partial^k_t u)^2
\le \frac{C_2 k^{3d/2+2}}{|B(x_0, 1)|}
\int_{Q_{1/\sqrt{k}} (0,x_0)} (\partial^k_t u)^2(t,x) \ dx dt.
\ee Next we will bound the right-hand side.

For integers $j=1, 2, \ldots, k,$ consider the domains:
$$
\al
\Omega^1_{j}&=\{(t,x) \, | \, d(x, x_0)<j/\sqrt{k}, \,\,  t
\in
[-j/k, 0] \},\\
\Omega^2_j&=\{(t,x) \, | \, d(x, x_0)<(j+0.5)/\sqrt{k}, \,\,
t
\in [-(j+0.5)/k, 0] \}.
\eal
$$
Then it is clear that $\Omega^1_j \subset \Omega^2_j \subset
\Omega^1_{j+1}$.

Denote by $\psi^{(1)}_j$ a standard Lipschitz cutoff function
supported in
$$
\{(t,x) \, | \, d(x, x_0)<(j+0.5)/\sqrt{k}, \,\,  t \in (-(j+0.5)/k, (j+0.5)/k)
\}
$$
such that $\psi^{(1)}_j=1$ in $\Omega^1_j$ and $|\nabla
\psi^{(1)}_j|^2 + |\partial_t \psi^{(1)}_j| \le C k$.

 Since $u$ is a smooth solution to the heat equation, we
 deduce,
 by writing
  $\psi=\psi^{(1)}_j$, that
\[
\al
\int_{\Omega^2_j} &(u_t)^2 \psi^2 \ dx dt = \int_{\Omega^2_j}
u_t \Delta u \psi^2 \ dx dt\\
&=-\int_{\Omega^2_j} ((\nabla u)_t \nabla u) \, \psi^2 \ dx dt
- \int_{\Omega^2_j} u_t \nabla u \nabla \psi^2 \ dx dt\\
&= - \frac{1}{2} \int_{\Omega^2_j} (|\nabla u |^2)_t \, \psi^2
\
dx dt - 2 \int_{\Omega^2_j} u_t \psi \nabla u \nabla \psi \ dx
dt\\
&\le \frac{1}{2} \int_{\Omega^2_j} |\nabla u |^2 \, (\psi^2)_t
\
dx dt +\frac{1}{2}
 \int_{\Omega^2_j} (u_t)^2 \psi^2 \ dx dt + 2 \int_{\Omega^2_j}
 |\nabla u|^2 |\nabla \psi|^2 \ dx dt.
\eal
\]
Therefore,
\be
\lab{omg1j<}
\int_{\Omega^1_j} ( u_t)^2 \psi^2 \ dx dt \le C k
\int_{\Omega^2_j} |\nabla u |^2 \ dx dt.
\ee

Denote by $\psi^{(2)}_j$ a standard Lipschitz cutoff function
supported in
$$
\{(t,x) \, | \, d(x, x_0)<(j+1)/\sqrt{k}, \quad  t \in
(-(j+1)/k, (j+1)/k) \}
$$
 such that $\psi^{(2)}_j=1$ in $\Omega^2_j$ and
$$
|\nabla
 \psi^{(2)}_j|^2 + |\partial_t \psi^{(2)}_j| \le C k.
$$
Using $(\psi^{(2)}_j)^2 u$ as a test function in the heat
equation, the standard Caccioppoli inequality (energy
estimate)
between the cubes $\Omega^2_j$ and  $\Omega^1_{j+1}$ shows
that
 \be
\lab{omg2j<}
\int_{\Omega^2_j} | \nabla u|^2 \ dx dt \le C k
\int_{\Omega^1_{j+1}}  u^2 \ dx dt.
\ee  A combination of (\ref{omg1j<}) and (\ref{omg2j<}) gives
us
\be
\lab{ddu2r-4}
\int_{\Omega^1_j} ( u_t)^2  \ dx dt \le C_0 k^2
\int_{\Omega^1_{j+1}} u^2  \ dx dt,
\ee where $C_0$ is a universal constant.

Since $\partial^j_t u$ is a solution, we can replace $u$ in
(\ref{ddu2r-4}) with $\partial^j_t u$ to deduce, after induction:
\be
\lab{utk<}
\int_{\Omega^1_1} (\partial^k_t u)^2  dx dt \le C^k_0 (k)^{2 k}
\int_{\Omega^1_k} u^2  \ dx dt.
\ee

Note that $\Omega^1_1= Q_{1/\sqrt{k}}(0,x_0)$ and
$\Omega^1_k=[-1, 0] \times B(x_0, \sqrt{k})$. Substituting
(\ref{utk<}) into (\ref{mviqdkt}), we
find that
\[
 (\partial^k_t u)^2(0,x_0) \le \frac{C_2 k^{3d/2+2}}{|B(x_0,
 1)|} C^k_0 (k)^{2 k} \int_{[-1, 0] \times  B(x_0, \sqrt{k}) }
 u^2  \ dx dt.
\]
This implies, by the  growth condition
(\ref{expg}) and the Bishop-Gromov volume comparison theorem,
that
\be
\lab{djtu}
|\partial^k_t u(0,x_0)| \le  A^{k+1}_3 k^k e^{ 2 A_2 d^2(x_0,
0)}
\ee for all integers $k \ge 1$. Here $A_3$ is a positive
constant depending only on $A_2$, $K_0$, and $d$ .
We remark that the volume of the ball in the denominator is
cancelled since
\[
|B(x_0, \sqrt{k})| \le e^{c \sqrt{k}} |B(x_0, 1)|.
\]

Fixing a number $R \ge 1$, for $x \in B(0, R)$, choose a
positive integer $j$  and $t \in [- \delta, 0]$
for some small $\delta>0$. Taylor's theorem implies that
\be
\lab{jtaylor}
u(t,x)- \sum^{j-1}_{i=0} \partial^i_t u(0,x) \frac{ t^i}{i !} =
\frac{t^j}{ j !} \partial^j_s u(s,x),
\ee where $s=s(x, t, j) \in [t, 0]$.
By (\ref{djtu}), for sufficiently small $\delta>0$, the
right-hand side of (\ref{jtaylor}) converges to $0$ uniformly
for $x \in B(0, R)$
as $j \to \infty$. Hence
$$
u(t,x)= \sum^{\infty}_{j=0} \partial^j_t u(0,x) \frac{ t^j}{j
!},
$$
i.e., $u$ is analytic in $t$ with radius $\delta$.  Writing
$a_j= a_j(x) = \partial^j_t u(0,x)$. By (\ref{djtu}) again, we
have
$$
\partial_t u(t,x) = \sum^{\infty}_{j=0} a_{j+1}(x) \frac{
t^j}{j
!}
\quad \text{and}\quad
\Delta u(t,x) = \sum^{\infty}_{j=0} \Delta a_j(x) \frac{ t^j}{j
!},
$$
where both series converge uniformly for $(t,x) \in
[-\delta, 0] \times B(0, R) $ for any fixed $R>0$.
Since $u$ is a solution of the heat equation,
this implies
\[
\Delta a_j(x) =a_{j+1}(x)
\] with
\[
|a_j(x)| \le A^{j+1}_3 j^j e^{2 A_2 d^2(x, 0)}.
\] Here $A_3$ a positive constant depending only on $A_2$,
$K_0$, and $d$. This completes the proof of the theorem.
\end{proof}
%\medskip

An immediate application is the following:

\begin{corollary}
\lab{cobhe}
Let $\M$ be as in the theorem. Then the Cauchy problem for the
backward heat equation
\be
\lab{bhe}
\begin{cases} \partial_t u+\Delta u  = 0,\\
u(0,x)=a(x)
\end{cases}
\ee has a smooth  solution of  exponential growth  of order $2$ in
$(0,\delta) \times {\M}$ for some $\delta>0$ if and only if
\be
\lab{aijie}
|\Delta^j a(x)| \le A^{j+1}_3 j^j e^{ A_4  d^2(x, 0)}, \quad
j=0, 1,2,\ldots,
\ee where  $A_3$ and $A_4$ are some positive constants.
\end{corollary}

\begin{proof}
Suppose (\ref{bhe}) has a smooth solution of  exponential
growth  of order $2$, say $u=u(t,x)$. Then
$u(x, -t)$ is a solution of the heat equation with  exponential
growth  of order $2$. By the theorem
\[
u(x, -t) = \sum^\infty_{j=0} a_j(x) \frac{(-t)^j}{j!}.
\] Then (\ref{aijie}) follows from the theorem since $\Delta^j
a(x) = a_j(x)$ in the theorem.

On the other hand, suppose (\ref{aijie}) holds. Then it is easy
to check that
\[
u(t,x) = \sum^\infty_{j=0} \Delta^j a(x) \frac{t^j}{j!}
\]is a smooth solution of the heat equation for $t \in
[-\delta,
0]$ with
$\delta$ sufficiently small. Indeed, the bounds (\ref{aijie})
guarantee that the above series and the series
\[
 \sum^\infty_{j=0} \Delta^{j+1} a(x) \frac{t^j}{j!} \quad\text{and}\quad
\sum^\infty_{j=0} \Delta^{j} a(x)  \frac{\partial_t t^j}{j!}
\]
all converge absolutely and uniformly in $[-\delta, 0]\times
B(0, R) $ for any fixed $R>0$. Hence $\partial_t u-\Delta u
=0$.
 Moreover $u$ has exponential
growth  of order $2$ since
 \[
|u(t,x)| \le \sum^\infty_{j=0} |\Delta^j a(x)| \frac{|t|^j}{j!}
\le A_3  e^{A_4 d^2(x, 0)} \sum^\infty_{j=0} \frac{\left(A_3  j
|t| \right)^j}{j!} \le A_3  e^{A_4 d^2(x, 0)}
\]
provided that $t \in [-\delta, 0]$ with $\delta$ sufficiently
small.
Thus $u(x, -t)$ is a solution to the Cauchy problem of the
backward heat equation (\ref{bhe}) of exponential
growth  of order $2$.
\end{proof}

{\remark  For the conclusion of the theorem to hold, some
growth
condition for the solution is necessary. Tychonov's
non-uniqueness example can be modified as follows. Let
$v=v(t,x)$ be Tychonov's solution of the heat equation in
$(-\infty, \infty)\times  {\R^d} $,  which is $0$ when $t \le
0$
but nontrivial for $t>0$.
Then $u \equiv v(x, t+1)$ is a nontrivial ancient solution. It
is clearly not analytic in time.
Note that $|u(t,x)|$ grows faster than $e^{c |x|^2}$  for any
$c>0$, but for any $\varepsilon>0$, $|u(x,t)|$ is bounded by
$Ce^{c|x|^{2+\varepsilon}}$ for some positive constants $c$ and
$C$. This implies that our growth condition is sharp}.

{\remark  If $\M =\R^d$,  it is well known that the
solution in the theorem is
also analytic in space variables. In fact, in this
case for the time analyticity the Laplace operator can be
replaced with a uniformly elliptic operator in divergence form
$D_i(a_{ij}D_j)$, where $a_{ij}$ are measurable functions
depending only on $x$. For general manifolds, the
space
analyticity requires certain bounds on curvature and its
derivatives.}

{\remark The time analyticity and backward solvability  results are still valid if the manifold $\M$ is replaced with
a bounded Lipschitz domain $D$ and the solutions are required to satisfy the Dirichlet or Neumann boundary condition. To justify, we can modify the proof of the theorem by replacing the geodesic balls with the fixed domain $D$ intersected with balls. The proof goes through for the following reasons. The first is that the mean value inequality with the geodesic balls replaced with $D$ intersected with balls still holds for $\partial^j_t u$ which satisfies the same boundary condition as $u$ itself. The second is that the boundary terms in the integration by parts procedure all vanish when the geodesic balls are replaced with $D$ intersected with balls.

All these results are also true with the appearance of a time independent inhomogeneous term, which is smooth and has at most exponential growth of order $2$, on the right-hand side of the heat equation. This may have applications in control theory.}

Next we present a further application of the main result in the section:  time analyticity of solutions of the heat equation down to the initial time $0$. Recall the following classical example from \cite{Kow:1}. Given the initial value $u_0(x) = 1/(1+ x^2)$ in $\R$, the solution to the Cauchy problem of the heat equation in $[0, \infty)\times \R $ is not analytic in time at $t=0$.
In the same paper, analyticity of solution down to $t=0$ is linked to the extension property of the initial value to an entire function in the complex plane with certain growth condition. See also \cite{Wi:1} Corollary 3.1b for a relatively modern approach. We mention that under the growth condition, time analyticity down to time $0$ also implies solvability of the backward heat equation, at least in short time. However it seems that the classical authors were mainly concerned with the concept of  analytic extension in space time without realizing that under the growth condition, solutions of the backward heat equation are analytic in time automatically. %Therefore the analytic extension is in fact the only extension {\color{red}(Extension in which sense? The meaning of this sentence is still unclear to me...)}.
The following result, which is an immediate consequence of the above corollary, provides a necessary and sufficient condition for solutions
to be analytic up to time $0$ without dependence on spatial analyticity, which is not available in the general setting.

\begin{corollary}
\lab{cojiexi0}
Let $\M$ be as in the theorem. Then the Cauchy problem for the heat equation
\be
\lab{he2}
\begin{cases} \partial_t u -\Delta u  = 0,\\
u(0,x)=a(x)
\end{cases}
\ee has a smooth  solution of exponential growth of order $2$, which is also analytic in time in
$[0, \delta) \times {\M}$ for some $\delta>0$ with a radius of convergence independent of $x$ if and only if
\be
\lab{aijie2}
|\Delta^j a(x)| \le A^{j+1}_3 j^j e^{ A_4  d^2(x, 0)}, \quad
j=0, 1,2,\ldots,
\ee where  $A_3$ and $A_4$ are some positive constants.
\end{corollary}

\proof In the proof, all solutions are of exponential growth of order $2$.

Assuming (\ref{aijie2}),  it is a standard result that
 problem (\ref{he2}) has a solution $u=u(t, x)$ for some $\delta>0$. By Corollary \ref{cobhe}, the following backward problem also has a solution
\be
\lab{bhev}
\begin{cases} \partial_t v +\Delta v  = 0,\\
v(0,x)=a(x)
\end{cases}
\ee in $[0,
\delta) \times {\M}.$ Define the function $U=U(t, x)$ by
\be
U(t, x)=\begin{cases}  u(t, x), \quad t \in [0, \delta),\\
v(-t, x), \quad t \in (-\delta, 0].
\end{cases}
\ee It is straight forward to check that $U$ is a solution of the heat equation in $(-\delta, \delta) \times \M$. By the theorem, $U$ and hence $u$ is analytic at time $0$. We mention that the time interval in the theorem is normalized to $[-2, 0]$ which can be changed to any finite interval.

On the other hand, suppose $u$ is a solution of the problem (\ref{he2}), which is analytic in time at $t=0$ with a radius of convergence independent of $x$. Then, by definition,  $u$ has a power series expansion in a time interval $(-\delta, \delta)$, for some $\delta>0$. Hence (\ref{aijie2}) holds following the proof of Corollary \ref{cobhe}, the first half.
\qed

\section{The Navier-Stokes equations}

The main result of this section is the following theorem.

\begin{theorem}
\label{thm2.1}
Assume that $u$ is a mild solution to the incompressible
Navier-Stokes equations
$$
u_t-\Delta u+u\cdot \nabla u+\nabla p=0
$$
on $[0,1]\times \bR^d$ and
\begin{equation}
                                \label{eq7.19}
|u|\le C_2 \quad \text{in}\ \ [0,1]\times \bR^d.
\end{equation}
Then for any $n\ge 1$,
\begin{equation*}
                %                    \label{eq2.48aa}
\sup_{t\in (0,1]} t^n \|\partial_t^n
u(t,\cdot)\|_{L_\infty(\bR^d)}
\le N^{n+1}n^{n}
\end{equation*}
for some sufficiently large constant $N\ge 1$. Consequently,
$u(t,x)$ is analytic in time for any $t\in (0,1]$.
\end{theorem}

 The proof of the theorem relies on taking time derivative of
 the integral representation of the solution involving the
 Stokes kernel. One difficulty to overcome is that the time
 derivative of the Stokes kernel is not locally integrable in
 space time, let alone higher order derivatives.
We will manipulate the kernel function algebraically to allow
differentiation. This method seems to be applicable to other
types of equations. We also mention that non-mild solutions of
the Navier-Stokes equations need not be analytic in time, as
given by Serrin's example $u = a(t) \nabla h(x)$ where $h$ is a
harmonic function and $a=a(t)$ is an arbitrary smooth function.

The following lemma will be used frequently.

\begin{lemma}
                                    \label{lem2.1}
For any $n\ge 1$, we have
$$
\sum_{j=1}^{n-1}\binom{n}{j}j^{j-2/3}(n-j)^{n-j-2/3}\le
Cn^{n-2/3},
$$
\end{lemma}
where $C>0$ is a constant independent of $n$.
\begin{proof}
By the Stirling formula,
\begin{align*}
&\sum_{j=1}^{n-1}\binom{n}{j}j^{j-2/3}(n-j)^{n-j-2/3}
\le Cn^{n-2/3} \sum_{j=1}^{n-1} \frac
{n^{7/6}}{j^{7/6}(n-j)^{7/6}}\\
&=Cn^{n-2/3} \sum_{j=1}^{n-1} \left(\frac 1 j+\frac 1 {n-j}
\right)^{7/6}\\
&\le Cn^{n-2/3} \sum_{j=1}^{n-1}\left( \frac 1 {j^{7/6}}+\frac 1
{(n-j)^{7/6}} \right)
\le Cn^{n-2/3}.
\end{align*}
The lemma is proved.
\end{proof}

The next combinatorial lemma can be proved by using induction.

\begin{lemma}
                                \label{lem2.2}
Let $f$ and $g$ be two smooth functions on $\bR$. For any
integer $n\ge 1$, we have
\begin{align*}
D^n(t^n f(t)g(t))&=\sum_{j=0}^n \binom{n}{j} D^j(t^j
f(t))D^{n-j}(t^{n-j}g(t))\\
&\quad -n\sum_{j=0}^{n-1} \binom{n-1}{j} D^j(t^j
f(t))D^{n-1-j}(t^{n-1-j}g(t)),
\end{align*}
where we denote $D=\partial_t$.
\end{lemma}
\begin{proof}
It follows from a straightforward computation by using
$$
D^n(t^n f(t)g(t))=tD^n(t^{n-1} f(t)g(t))+nD^{n-1}(t^{n-1}
f(t)g(t))
$$
and the inductive assumption.
\end{proof}

Let $\mathcal P$ be the Helmholtz (Leray--Hopf) projection in
$\bR^d$, and
$E(t,x)=\mathcal P\Gamma(t,x)$ be the Stokes-Oseen kernel,
where
$$
\Gamma=(4\pi t)^{-d/2}e^{-|x|^2/(4t)}
$$
is
the heat kernel. Recall that $E$ satisfies the
homogeneous heat equation, the semigroup property, and
$$
E(t,x)=t^{-d/2}E(1,x/\sqrt t),
$$
where $E(1,\cdot)$ is a smooth function on $\bR^d$ and decays
like $C/|x|^d$ as $x\to \infty$. Moreover,
$(\partial_t E)(1,x)$ and $(\nabla E)(1,x)$ decay like
$C/|x|^{d+2}$ and $C/|x|^{d+1}$ respectively as $x\to \infty$.
See, for instance, \cite{So}. Using these properties, we easily
obtain
\begin{equation}
                                \label{eq1.56a}
\|\nabla E(t,\cdot)\|_{L_1}\le C_0 t^{-1/2},
\end{equation}
and
\begin{equation*}
%                                \label{eq1.56}
\|\partial_t^k E(t,\cdot)\|_{L_1}\le C_0^{k+1}k^{k-2/3}
t^{-k},\quad
\|\partial_t^k \nabla E(t,\cdot)\|_{L_1}\le C_0^{k+1}k^{k-2/3}
t^{-k-1/2}
\end{equation*}
for any integer $k\ge 1$ and $t>0$, where $C_0\ge 1$ is a
constant. It then follows from the Leibniz rule that
\begin{equation}
                                \label{eq1.56b}
%\|\partial_t^k (t^k E(t,\cdot))\|_{L_1}\le C_1^{k+1}k^{k-2/3}
\|\partial_t^k (t^k \nabla E(t,\cdot))\|_{L_1}\le
C_1^{k+1}k^{k-2/3} t^{-1/2}.
\end{equation}
Similarly, we have
\begin{equation}
                                \label{eq1.56c}
\|\partial_t^k (t^k \Gamma(t,\cdot))\|_{L_1}\le
C_1^{k+1}k^{k-2/3}.
\end{equation}

To prove Theorem \ref{thm2.1}, we first establish the following
proposition.
\begin{proposition}
                                    \label{prop2.2}
Under the conditions of Theorem \ref{thm2.1}, for any $n\ge 1$,
we have
\begin{equation}
                                    \label{eq2.48}
\sup_{t\in (0,1]} \|\partial_t^n (t^n
u(t,\cdot))\|_{L_\infty(\bR^d)}
\le N^{n-1/2}n^{n-2/3}
\end{equation}
for some sufficiently large constant $N\ge 1$ depending only on the dimension and $\| u \|_{L_\infty}$.
\end{proposition}
\begin{proof}
We shall prove the proposition inductively.
As $u$ is a mild solution, we have
$$
u(t,x)=E(t,x)*u(0,x)-\int_0^t E(t-s,x)*\nabla (u\otimes
u)(s,x)\,ds,
$$
where $*$ denotes the spatial convolution. Then, by using
integration by parts,
\begin{align*}
&\partial_t^n (t^n u(t,x))\\
&=\partial_t^n (t^n
E(t,x)*u(0,x))-\partial_t^n \left(\int_0^{t} t^n \nabla
E(t-s,x)*(u\otimes u)(s,x)\,ds\right)\\
&:=I_1+I_2.
\end{align*}
By using \eqref{eq7.19} and \eqref{eq1.56c},
\begin{equation*}
|I_1|\le C_2C_1^{n+1}n^{n-2/3}\le N^{n-2/3}n^{n-2/3}
\end{equation*}
for sufficiently large $N$.

To estimate $I_2$, we first note that
\begin{align*}
&\int_0^{t} t^n \nabla E(t-s,x)*(u\otimes u)(s,x)\,ds\\
&=\sum_{k=0}^n \binom{n}{k} \int_0^{t} (t-s)^{k} \nabla
E(t-s,x)*\left(s^{n-k}(u\otimes u)(s,x)\right)\,ds.
\end{align*}
Therefore,
\begin{align*}
I_2&=-\sum_{k=0}^n \binom{n}{k} \partial_t^{n} \int_0^{t}
\big((t-s)^{k} \nabla E(t-s,x)\big)*\left(s^{n-k}(u\otimes
u)(s,x)\right)\,ds\\
&=-\sum_{k=0}^n \binom{n}{k} \partial_t^{n-k} \int_0^{t}
\partial_t^k\big((t-s)^{k} \nabla
E(t-s,x)\big)*\left(s^{n-k}(u\otimes u)(s,x)\right)\,ds\\
&=-\sum_{k=0}^n \binom{n}{k} \partial_t^{n-k} \int_0^{t}
\partial_s^k\big(s^{k} \nabla
E(s,x)\big)*\left((t-s)^{n-k}(u\otimes u)(t-s,x)\right)\,ds\\
&=-\sum_{k=0}^n \binom{n}{k} \int_0^{t} \partial_s^k\big(s^{k}
\nabla E(s,x)\big)*\partial_t^{n-k} \left((t-s)^{n-k}(u\otimes
u)(t-s,x)\right)\,ds,
\end{align*}
where in the third equality, we made a change of time variable $s \to t-s$, which allows us to
pass the time derivatives $\partial^{n-k}_t$ through the integral  without producing additional terms.

When $n=1$, we compute
\[
\begin{aligned}
\big|\partial_t &( t (u \otimes u)(t,x) ) \big| \\
&= \big|\partial_t  (t u) \otimes u (t, x) + u \otimes \partial_t (t u) (t, x) - u \otimes u (t, x) \big|
\\
& \le C \| u \|_{L_\infty} \big| \partial_t (t u)(t, x) \big| +
C   \| u \|^2_{L_\infty},
\end{aligned}
\] where $C>0$ depends only on the dimension $d$.
In general, by the inductive assumption and Lemmas \ref{lem2.2} and
\ref{lem2.1}, we have for $k=1,\ldots,n-1$,
\begin{equation}
                                \label{eq7.29}
|\partial_t^k (t^k (u\otimes u)(t,x))|
\le N^{k-1/3}k^{k-2/3}
\end{equation}
and
\begin{equation}
                                    \label{eq8.00}
|\partial_t^n (t^n (u\otimes u)(t,x))|
\le 2C_2 |\partial_t^n (t^n u(t,x))|+N^{n-3/4}n^{n-2/3}
\end{equation}
provided that $N$ is sufficiently large, depending only on the dimension $d$ and
$\Vert u \Vert_{L_\infty}$.
It then follows from Lemma \ref{lem2.1}, \eqref{eq1.56a},
\eqref{eq1.56b}, \eqref{eq8.00}, and \eqref{eq7.29} that
\begin{align*}
|I_2|&\le \int_0^t (t-s)^{-1/2}\Big[C_1^{n+1} n^{n-2/3}C_2^2
+C_0\big(2C_2\|\partial_t^n((t-s)^n
u(t-s,\cdot))\|_{L_\infty}\\
&\quad + N^{n-3/4}n^{n-2/3}\big)
+\sum_{k=1}^{n-1}\binom{n}{k}C_1^{k+1}k^{k-2/3} \cdot
N^{n-k-1/3}(n-k)^{n-k-2/3}\Big]\ ds\\
&\le N^{n-2/3}n^{n-2/3}t^{1/2}+2C_1C_2\int_0^t
s^{-1/2}\|\partial_s^n(s^n u(s,\cdot))\|_{L_\infty}
\ ds
\end{align*}
for sufficiently large $N$ depending on $C_1$ and $C_2$.

Combining the estimates of $I_1$ and $I_2$, we get
\eqref{eq2.48} by applying the Gronwall inequality, and
complete
the proof of the proposition.
\end{proof}

Finally, we complete the proof of Theorem \ref{thm2.1}.
\begin{proof}[Proof of Theorem \ref{thm2.1}]
Note that
$$
\partial_t^n(t^k u)=n\partial_t^{n-1}(t^{k-1}
u)+t\partial_t^n(t^{k-1}u).
$$
Taking $k=n$ and using \eqref{eq2.48}, we obtain
$$
\sup_{t\in (0,1]} \|t\partial_t^n (t^{n-1}
u(t,\cdot))\|_{L_\infty(\bR^d)}
\le N^{n}(1+1/N)n^{n}.
$$
By induction,
$$
\sup_{t\in (0,1]} \|t^n\partial_t^n
u(t,\cdot)\|_{L_\infty(\bR^d)}
\le N^{n}(1+1/N)^n n^{n}=(N+1)^n n^{n}.
$$
The theorem is proved.
\end{proof}

\section*{Acknowledgement} We wish to thank Prof. Bobo Hua,
Igor Kukavica, and F. H. Lin for helpful discussions.

\medskip

\noindent {\bf Funding:} Hongjie Dong was partially supported by the NSF under agreement DMS-1600593. Qi S. Zhang was partially supported by the Simons Foundation Grant No. 282153.

\end{document}